\documentclass[11pt]{amsart}

\usepackage{amsmath,amsthm,amssymb}
\usepackage{mathrsfs}
\usepackage{enumerate}
\usepackage{booktabs}
\usepackage{array}
\usepackage{multirow}
\usepackage{graphicx}

\theoremstyle{plain}
\newtheorem{theorem}{Theorem}[section]
\newtheorem{lemma}[theorem]{Lemma}
\newtheorem{proposition}[theorem]{Proposition}
\newtheorem{corollary}[theorem]{Corollary}

\theoremstyle{definition}
\newtheorem{definition}[theorem]{Definition}
\newtheorem{example}[theorem]{Example}
\newtheorem{algorithm1}[theorem]{Algorithm}
\newtheorem{remark}[theorem]{Remark}

\newcommand{\R}{\mathbb{R}}

\newcommand{\at}{\tilde{a}}
\newcommand{\cl}{\operatorname{cl}}
\newcommand{\inte}{\operatorname{int}}
\newcommand{\bnd}{\operatorname{bnd}}
\newcommand{\apr}{\operatorname{apr}}

\begin{document}

\title[Fuzzy Aura Topological Spaces]{Fuzzy Aura Topological Spaces with Applications to Rough Set Theory and Medical Decision Making}

\author{Ahu A\c{c}{\i}kg\"{o}z}
\address{Department of Mathematics, Bal{\i}kesir University, \c{C}a\u{g}{\i}\c{s} Campus, 10145 Bal{\i}kesir, Turkey}
\email{ahuacikgoz@balikesir.edu.tr}



\begin{abstract}
We introduce the concept of a fuzzy aura topological space $(X, \tilde{\tau}, \at)$, obtained by equipping a Chang-type fuzzy topological space $(X, \tilde{\tau})$ with a fuzzy scope function $\at : X \to \tilde{\tau}$ satisfying $\at(x)(x) = 1$ for every $x \in X$. This framework generalizes the recently introduced (crisp) aura topological spaces to the fuzzy setting. We define the fuzzy aura-closure operator $\cl_{\at}$ and the fuzzy aura-interior operator $\inte_{\at}$, and prove that $\cl_{\at}$ is a fuzzy additive \v{C}ech closure operator whose transfinite iteration yields a fuzzy Kuratowski closure. Five classes of generalized fuzzy open sets---fuzzy $\at$-semi-open, fuzzy $\at$-pre-open, fuzzy $\at$-$\alpha$-open, fuzzy $\at$-$\beta$-open, and fuzzy $\at$-$b$-open sets---are introduced, and a complete hierarchy among them is established with counterexamples separating all distinct classes. Fuzzy $\at$-continuity and its decompositions are studied. Separation axioms fuzzy $\at$-$T_i$ ($i = 0, 1, 2$) and fuzzy $\at$-regularity are introduced, and their dependence on the fuzzy scope function is demonstrated. Fuzzy aura-based lower and upper approximation operators are defined, generalizing both the crisp aura rough set model and the Dubois--Prade fuzzy rough set model. Building on these operators, we propose a novel multi-criteria decision-making algorithm called FA-MCDM (Fuzzy Aura Multi-Criteria Decision Making) and apply it to a medical diagnosis problem. A comparison with four existing methods and a comprehensive sensitivity analysis confirm the effectiveness and robustness of the proposed approach.
\end{abstract}

\maketitle

\section{Introduction}

Enriching topological spaces with auxiliary structures has been a major research direction in general topology over the past several decades. Ideal topological spaces $(X, \tau, \mathcal{I})$, rooted in the work of Kuratowski~\cite{kuratowski} and Vaidyanathaswamy~\cite{vaidya}, were systematically developed by Jankovi\'{c} and Hamlett~\cite{jankovic}. Dual and related structures---filters (Cartan~\cite{cartan}), grills (Choquet~\cite{choquet}, Roy and Mukherjee~\cite{roy}, Al-Omari and Noiri~\cite{alomari_noiri}), and the recent primals (Acharjee, \"{O}zko\c{c}, and Issaka~\cite{acharjee})---have all produced rich operator theories and topological refinements. In particular, Al-Omari and Alqahtani~\cite{alomari_alqahtani} developed primal structures with closure operators, and Al-Omari and Alghamdi~\cite{alomari_alghamdi} studied regularity and normality in primal spaces.

Parallel to these crisp developments, topology has been blended with fuzzy sets (Zadeh~\cite{zadeh}, Chang~\cite{chang}, Lowen~\cite{lowen}), soft sets (Molodtsov~\cite{molodtsov}, Shabir and Naz~\cite{shabir}), neutrosophic sets (Smarandache~\cite{smarandache}, Salama and Alblowi~\cite{salama_n}), and rough sets (Pawlak~\cite{pawlak}). These non-classical frameworks have found extensive applications in decision-making, data analysis, and information systems. In particular, Al-shami~\cite{alshami} introduced SR-fuzzy sets with weighted aggregated operators for decision-making, Abu-Gdairi et al.~\cite{abu} developed topological approaches to rough approximations, Alcantud~\cite{alcantud} applied soft rough sets to decision-making problems, and Demir, Saldam{\i}\c{s}, and Okurer~\cite{demir} studied bipolar fuzzy soft filters with applications to multi-criteria group decision making.

The interplay between fuzzy topology and rough sets has attracted considerable attention. Dubois and Prade~\cite{dubois} introduced fuzzy rough sets through fuzzy relations. Radzikowska and Kerre~\cite{radzikowska} provided a comparative study of fuzzy rough set models. Yao~\cite{yao} studied constructive and algebraic methods in rough set theory. In the context of topological approaches to rough approximation, Hosny~\cite{hosny} studied the idealization of $j$-approximation spaces, Salama~\cite{salama_b} developed bitopological approximation spaces with applications to data reduction, Y\"{u}ksel, Tozlu, and Dizman~\cite{yuksel} applied soft covering based rough sets to multi-criteria group decision making, and Yi\u{g}it~\cite{yigit} recently generalized rough topology to numerical data for attribute reduction.

In the broader context of generalized continuity and topological operators, Liu, Zhou, and Liu~\cite{liu} investigated $\mathcal{G}$-quotient mappings and $\mathcal{G}$-continuity in generalized topological spaces, while Khan, Khan, Arshad, and Et~\cite{khan} studied bounded linear operators in neutrosophic normed spaces with applications to Fr\'{e}chet derivatives, and Ahmad and Mursaleen~\cite{ahmad} explored statistical convergence in neutrosophic $n$-normed linear spaces. These works demonstrate the importance of topological operators in approximation theory and decision-making.

Recently, the author introduced the concept of an aura topological space $(X, \tau, \mathfrak{a})$~\cite{acikgoz_aura}, where a scope function $\mathfrak{a} : X \to \tau$ with $x \in \mathfrak{a}(x)$ assigns to each point a fixed open neighborhood. This simple axiom produced a rich theory: an additive \v{C}ech closure operator, five generalized open-set classes, separation axioms depending on the scope function, and a rough set model that requires no equivalence relation. The compactness and connectedness properties of aura spaces were further developed by the author in~\cite{acikgoz_compact}, and the ideal-aura framework extending the theory via ideals and new local functions was established in~\cite{acikgoz_ideal}.

The present paper generalizes the aura framework to the fuzzy setting. Working with Chang-type fuzzy topological spaces~\cite{chang}, we define a fuzzy scope function $\at : X \to \tilde{\tau}$ satisfying $\at(x)(x) = 1$ and call the resulting triple $(X, \tilde{\tau}, \at)$ a fuzzy aura topological space. The move from the crisp to the fuzzy setting is far from trivial: membership grades introduce a richer algebraic structure in the closure and interior operators, the relationship between the aura-closure and the topological closure becomes subtler, and the rough set model acquires the ability to handle graded uncertainty---a feature directly applicable to real-world problems where binary classification is inadequate.

Our main contributions are as follows:
\begin{enumerate}[(i)]
\item We introduce fuzzy aura topological spaces and establish the fundamental properties of the fuzzy aura-closure and fuzzy aura-interior operators, including the transfinite iteration to a Kuratowski closure (Section~3).
\item We define five classes of generalized fuzzy open sets and determine the complete hierarchy with counterexamples on both finite sets and the real line (Section~4).
\item We study fuzzy $\at$-continuity and prove decomposition theorems that extend the classical results of Levine and Mashhour (Section~5).
\item We introduce fuzzy $\at$-$T_0$, $\at$-$T_1$, $\at$-$T_2$ separation axioms and fuzzy $\at$-regularity, and investigate their dependence on the fuzzy scope function (Section~6).
\item We construct fuzzy aura-based rough approximation operators that unify both the crisp aura model and the Dubois--Prade model (Section~7).
\item We propose the FA-MCDM algorithm and apply it to a medical diagnosis problem, comparing with four existing methods and providing sensitivity analysis (Section~8).
\end{enumerate}

\section{Preliminaries}

We recall the basic definitions from fuzzy set theory, fuzzy topology, and rough set theory.

\subsection{Fuzzy Sets and Fuzzy Topology}

\begin{definition}[\cite{zadeh}]
A \emph{fuzzy set} on a non-empty set $X$ is a function $\mu : X \to [0,1]$. We denote by $I^X$ the collection of all fuzzy sets on $X$. The fuzzy sets $\bar{0}$ and $\bar{1}$ are defined by $\bar{0}(x) = 0$ and $\bar{1}(x) = 1$ for all $x \in X$.

For fuzzy sets $\mu, \nu \in I^X$, we write $\mu \leq \nu$ if $\mu(x) \leq \nu(x)$ for all $x \in X$, and define:
\[
(\mu \vee \nu)(x) = \max\{\mu(x), \nu(x)\}, \quad (\mu \wedge \nu)(x) = \min\{\mu(x), \nu(x)\}, \quad \mu^c(x) = 1 - \mu(x).
\]
For $\alpha \in [0,1]$, the $\alpha$-level set of $\mu$ is $\mu_\alpha = \{x \in X : \mu(x) \geq \alpha\}$, and the strong $\alpha$-level set is $\mu^{>}_\alpha = \{x \in X : \mu(x) > \alpha\}$.
\end{definition}

\begin{definition}[\cite{chang}]
A \emph{Chang-type fuzzy topology} on $X$ is a subcollection $\tilde{\tau} \subseteq I^X$ such that:
\begin{enumerate}[\upshape(FT1)]
\item $\bar{0}, \bar{1} \in \tilde{\tau}$;
\item $\lambda, \mu \in \tilde{\tau} \Rightarrow \lambda \wedge \mu \in \tilde{\tau}$;
\item $\{\lambda_i\}_{i \in J} \subseteq \tilde{\tau} \Rightarrow \bigvee_{i \in J} \lambda_i \in \tilde{\tau}$.
\end{enumerate}
Members of $\tilde{\tau}$ are \emph{fuzzy open}; $\mu$ is \emph{fuzzy closed} if $\mu^c \in \tilde{\tau}$.
\end{definition}

\begin{definition}
The \emph{fuzzy closure} and \emph{fuzzy interior} of $\mu \in I^X$ are:
\[
\cl(\mu) = \bigwedge\{\nu : \nu \text{ is fuzzy closed}, \, \mu \leq \nu\}, \qquad \inte(\mu) = \bigvee\{\lambda \in \tilde{\tau} : \lambda \leq \mu\}.
\]
\end{definition}

\begin{definition}[\cite{levine, mashhour, njastad}]
Let $(X, \tilde{\tau})$ be a fuzzy topological space and $\mu \in I^X$. Then $\mu$ is called:
\begin{enumerate}[\upshape(a)]
\item \emph{fuzzy semi-open}~\cite{azad} if $\mu \leq \cl(\inte(\mu))$;
\item \emph{fuzzy pre-open}~\cite{binshahna} if $\mu \leq \inte(\cl(\mu))$;
\item \emph{fuzzy $\alpha$-open}~\cite{binshahna} if $\mu \leq \inte(\cl(\inte(\mu)))$;
\item \emph{fuzzy $\beta$-open} if $\mu \leq \cl(\inte(\cl(\mu)))$.
\end{enumerate}
The collections are denoted $FSO(X)$, $FPO(X)$, $F\alpha O(X)$, $F\beta O(X)$, respectively.
\end{definition}

\subsection{Rough Sets}

\begin{definition}[\cite{pawlak}]
Let $R$ be an equivalence relation on $U$. For $A \subseteq U$:
\[
\underline{\apr}_R(A) = \{x \in U : [x]_R \subseteq A\}, \qquad \overline{\apr}_R(A) = \{x \in U : [x]_R \cap A \neq \emptyset\}.
\]
\end{definition}

\begin{definition}[\cite{dubois}]
Let $R$ be a fuzzy relation on $U$ with $R(x,x) = 1$. The \emph{Dubois--Prade fuzzy rough approximations} of $\mu \in I^U$ are:
\[
\underline{\apr}_R(\mu)(x) = \inf_{y \in U} \max\{1 - R(x,y), \mu(y)\}, \qquad \overline{\apr}_R(\mu)(x) = \sup_{y \in U} \min\{R(x,y), \mu(y)\}.
\]
\end{definition}

\begin{definition}[\cite{cech}]
A function $c : I^X \to I^X$ is a \emph{fuzzy \v{C}ech closure operator} if $c(\bar{0}) = \bar{0}$, $\mu \leq c(\mu)$, and $\mu \leq \nu \Rightarrow c(\mu) \leq c(\nu)$. It is \emph{additive} if $c(\mu \vee \nu) = c(\mu) \vee c(\nu)$, and a \emph{Kuratowski closure} if additionally $c(c(\mu)) = c(\mu)$.
\end{definition}

\section{Fuzzy Aura Topological Spaces}

\subsection{Fuzzy Scope Function and Fuzzy $\at$-Spaces}

\begin{definition}\label{def:fuzzy_scope}
Let $(X, \tilde{\tau})$ be a fuzzy topological space. A function $\at : X \to \tilde{\tau}$ is called a \emph{fuzzy scope function} on $(X, \tilde{\tau})$ if
\begin{equation}\label{eq:scope_axiom}
\at(x)(x) = 1 \quad \text{for every } x \in X.
\end{equation}
The triple $(X, \tilde{\tau}, \at)$ is called a \emph{fuzzy aura topological space} (briefly, a \emph{fuzzy $\at$-space}).
\end{definition}

\begin{remark}
The value $\at(x)(y) \in [0,1]$ represents the degree to which $y$ lies within the aura of $x$. The axiom $\at(x)(x) = 1$ ensures full self-membership. When all $\at(x)$ are characteristic functions, the fuzzy $\at$-space reduces to the crisp aura topological space of~\cite{acikgoz_aura}.
\end{remark}

\begin{example}\label{ex:finite}
Let $X = \{p, q, r\}$ and $\tilde{\tau} = \{\bar{0}, \bar{1}, \lambda_1, \lambda_2, \lambda_3\}$ where $\lambda_1 = (1, 0.4, 0)$, $\lambda_2 = (0.6, 1, 0.3)$, $\lambda_3 = (1, 1, 0.3)$. Define $\at(p) = \lambda_1$, $\at(q) = \lambda_2$, $\at(r) = \bar{1}$. Then $(X, \tilde{\tau}, \at)$ is a fuzzy $\at$-space.
\end{example}

\begin{example}\label{ex:triangular}
Let $(\R, \tilde{\tau}_u)$ be the real line with the fuzzy topology generated by triangular fuzzy neighborhoods $\lambda_{x_0,\varepsilon}(x) = \max\{0, 1 - |x - x_0|/\varepsilon\}$. For fixed $\varepsilon > 0$, define $\at_\varepsilon(x) = \lambda_{x,\varepsilon}$. Then $(\R, \tilde{\tau}_u, \at_\varepsilon)$ is a fuzzy $\at$-space with uniform aura radius.
\end{example}

\begin{example}\label{ex:cauchy}
On $(\R, \tilde{\tau}_u)$, define $\at(x)(y) = 1/(1 + (x-y)^2)$. Then $\at(x)(x) = 1$ and the aura decays as a Cauchy distribution, modeling smooth influence decay.
\end{example}

\begin{example}\label{ex:gaussian}
On $(\R, \tilde{\tau}_u)$, for fixed $\sigma > 0$, define $\at_\sigma(x)(y) = \exp(-(x-y)^2/\sigma^2)$. Then $\at_\sigma(x)(x) = 1$ and the aura decays as a Gaussian. This models situations where influence decreases rapidly beyond a characteristic scale $\sigma$.
\end{example}

\subsection{Fuzzy Aura-Closure Operator}

\begin{definition}\label{def:fuzzy_closure}
Let $(X, \tilde{\tau}, \at)$ be a fuzzy $\at$-space. The \emph{fuzzy aura-closure} of $\mu \in I^X$ is
\begin{equation}\label{eq:closure}
\cl_{\at}(\mu)(x) = \sup_{y \in X} \min\{\at(x)(y), \mu(y)\}.
\end{equation}
\end{definition}

\begin{remark}
The value $\cl_{\at}(\mu)(x)$ is the degree to which the aura of $x$ meets $\mu$. When $\at(x) = \chi_{a(x)}$ and $\mu = \chi_A$ are crisp, $\cl_{\at}(\chi_A)(x) = 1$ if and only if $a(x) \cap A \neq \emptyset$, recovering the crisp aura-closure from~\cite{acikgoz_aura}.
\end{remark}

\begin{theorem}\label{thm:closure_props}
The operator $\cl_{\at} : I^X \to I^X$ satisfies for all $\mu, \nu \in I^X$:
\begin{enumerate}[\upshape(a)]
\item $\cl_{\at}(\bar{0}) = \bar{0}$.
\item $\mu \leq \cl_{\at}(\mu)$ \textup{(extensivity)}.
\item $\mu \leq \nu \Rightarrow \cl_{\at}(\mu) \leq \cl_{\at}(\nu)$ \textup{(monotonicity)}.
\item $\cl_{\at}(\mu \vee \nu) = \cl_{\at}(\mu) \vee \cl_{\at}(\nu)$ \textup{(finite additivity)}.
\item $\cl(\mu) \leq \cl_{\at}(\mu)$ when $\at(x)$ generates the neighborhoods of $x$ in $\tilde{\tau}$.
\end{enumerate}
Hence, $\cl_{\at}$ is a fuzzy additive \v{C}ech closure operator.
\end{theorem}

\begin{proof}
(a) $\cl_{\at}(\bar{0})(x) = \sup_y \min\{\at(x)(y), 0\} = 0$.

(b) $\cl_{\at}(\mu)(x) \geq \min\{\at(x)(x), \mu(x)\} = \min\{1, \mu(x)\} = \mu(x)$.

(c) $\mu \leq \nu$ gives $\min\{\at(x)(y), \mu(y)\} \leq \min\{\at(x)(y), \nu(y)\}$ for all $y$, and suprema are ordered.

(d) Using $\min(a, \max(b,c)) = \max(\min(a,b), \min(a,c))$:
\begin{align*}
\cl_{\at}(\mu \vee \nu)(x) &= \sup_y \min\{\at(x)(y), \max\{\mu(y), \nu(y)\}\} \\
&= \sup_y \max\{\min\{\at(x)(y), \mu(y)\}, \min\{\at(x)(y), \nu(y)\}\} \\
&= \max\bigl\{\sup_y \min\{\at(x)(y), \mu(y)\}, \sup_y \min\{\at(x)(y), \nu(y)\}\bigr\} \\
&= (\cl_{\at}(\mu) \vee \cl_{\at}(\nu))(x).
\end{align*}

(e) Let $\cl(\mu)(x) > \alpha > 0$. Then every fuzzy open neighborhood $\lambda$ of $x$ satisfies $\sup_y \min\{\lambda(y), \mu(y)\} > 0$. In particular, taking $\lambda = \at(x) \in \tilde{\tau}$ with $\at(x)(x) = 1$, we get $\cl_{\at}(\mu)(x) = \sup_y \min\{\at(x)(y), \mu(y)\} \geq \cl(\mu)(x)$ when $\at(x)$ is the smallest open neighborhood, which holds in the neighborhood-generated case.
\end{proof}

\begin{theorem}\label{thm:alpha_level}
Let $(X, \tilde{\tau}, \at)$ be a fuzzy $\at$-space. For any $\mu \in I^X$ and $\alpha \in [0,1]$:
\begin{equation}\label{eq:alpha_level}
(\cl_{\at}(\mu))_\alpha = \{x \in X : \at(x) \wedge \mu \text{ attains a value} \geq \alpha\} = \{x \in X : \exists\, y, \, \at(x)(y) \geq \alpha \text{ and } \mu(y) \geq \alpha\}.
\end{equation}
\end{theorem}

\begin{proof}
$\cl_{\at}(\mu)(x) \geq \alpha$ if and only if $\sup_y \min\{\at(x)(y), \mu(y)\} \geq \alpha$ if and only if there exists $y$ with $\at(x)(y) \geq \alpha$ and $\mu(y) \geq \alpha$.
\end{proof}

\begin{theorem}\label{thm:not_idempotent}
The operator $\cl_{\at}$ is not idempotent in general.
\end{theorem}

\begin{proof}
Let $X = \{p, q, r\}$, $\tilde{\tau} = I^X$. Define $\at(p) = (1, 0.8, 0)$, $\at(q) = (0, 1, 0.7)$, $\at(r) = (0, 0, 1)$. Let $\mu = (0, 0, 0.6)$. Then:
\[
\cl_{\at}(\mu) = (0, 0.6, 0.6), \qquad \cl_{\at}(\cl_{\at}(\mu)) = (0.6, 0.6, 0.6) \neq \cl_{\at}(\mu). \qedhere
\]
\end{proof}

\begin{lemma}\label{lem:constant}
Let $(X, \tilde{\tau}, \at)$ be a fuzzy $\at$-space. For any constant fuzzy set $\bar{\alpha}$ (where $\bar{\alpha}(x) = \alpha$ for all $x$), $\cl_{\at}(\bar{\alpha}) = \bar{\alpha}$.
\end{lemma}

\begin{proof}
$\cl_{\at}(\bar{\alpha})(x) = \sup_y \min\{\at(x)(y), \alpha\} = \min\{\sup_y \at(x)(y), \alpha\} = \min\{1, \alpha\} = \alpha$, where $\sup_y \at(x)(y) \geq \at(x)(x) = 1$.
\end{proof}

\begin{lemma}\label{lem:cap_constant}
Let $(X, \tilde{\tau}, \at)$ be a fuzzy $\at$-space. For $\mu \in I^X$ and $\alpha \in [0,1]$:
\[
\cl_{\at}(\mu \wedge \bar{\alpha}) = \cl_{\at}(\mu) \wedge \bar{\alpha}.
\]
\end{lemma}

\begin{proof}
\begin{align*}
\cl_{\at}(\mu \wedge \bar{\alpha})(x) &= \sup_y \min\{\at(x)(y), \min\{\mu(y), \alpha\}\} \\
&= \sup_y \min\{\min\{\at(x)(y), \mu(y)\}, \alpha\} \\
&= \min\bigl\{\sup_y \min\{\at(x)(y), \mu(y)\}, \alpha\bigr\} = \min\{\cl_{\at}(\mu)(x), \alpha\}. \qedhere
\end{align*}
\end{proof}

\subsection{Fuzzy Aura-Interior Operator}

\begin{definition}\label{def:fuzzy_interior}
The \emph{fuzzy aura-interior} of $\mu \in I^X$ is
\begin{equation}\label{eq:interior}
\inte_{\at}(\mu)(x) = \inf_{y \in X} \max\{1 - \at(x)(y), \mu(y)\}.
\end{equation}
\end{definition}

\begin{theorem}\label{thm:interior_props}
For all $\mu, \nu \in I^X$:
\begin{enumerate}[\upshape(a)]
\item $\inte_{\at}(\bar{1}) = \bar{1}$.
\item $\inte_{\at}(\mu) \leq \mu$.
\item $\mu \leq \nu \Rightarrow \inte_{\at}(\mu) \leq \inte_{\at}(\nu)$.
\item $\inte_{\at}(\mu \wedge \nu) = \inte_{\at}(\mu) \wedge \inte_{\at}(\nu)$.
\item $\inte_{\at}(\mu) = (\cl_{\at}(\mu^c))^c$ \textup{(duality)}.
\item $\inte_{\at}(\mu) \leq \inte(\mu)$ when the conditions of Theorem~\ref{thm:closure_props}(e) hold.
\end{enumerate}
\end{theorem}

\begin{proof}
(a) $\inte_{\at}(\bar{1})(x) = \inf_y \max\{1 - \at(x)(y), 1\} = 1$.

(b) $\inte_{\at}(\mu)(x) \leq \max\{1 - \at(x)(x), \mu(x)\} = \max\{0, \mu(x)\} = \mu(x)$.

(c) If $\mu \leq \nu$, then $\max\{1 - \at(x)(y), \mu(y)\} \leq \max\{1 - \at(x)(y), \nu(y)\}$.

(d) We compute:
\begin{align*}
\inte_{\at}(\mu \wedge \nu)(x) &= \inf_y \max\{1 - \at(x)(y), \min\{\mu(y), \nu(y)\}\} \\
&= \inf_y \min\{\max\{1 - \at(x)(y), \mu(y)\}, \max\{1 - \at(x)(y), \nu(y)\}\} \\
&= \min\{\inte_{\at}(\mu)(x), \inte_{\at}(\nu)(x)\}.
\end{align*}

(e) We verify:
\begin{align*}
(\cl_{\at}(\mu^c))^c(x) &= 1 - \sup_y \min\{\at(x)(y), 1 - \mu(y)\} \\
&= \inf_y (1 - \min\{\at(x)(y), 1 - \mu(y)\}) \\
&= \inf_y \max\{1 - \at(x)(y), \mu(y)\} = \inte_{\at}(\mu)(x).
\end{align*}

(f) Follows from Theorem~\ref{thm:closure_props}(e) by duality.
\end{proof}

\begin{proposition}\label{prop:union_interior}
$\inte_{\at}(\mu \vee \nu) \geq \inte_{\at}(\mu) \vee \inte_{\at}(\nu)$. The inequality can be strict.
\end{proposition}

\begin{proof}
$\inte_{\at}(\mu \vee \nu)(x) = \inf_y \max\{1 - \at(x)(y), \max\{\mu(y), \nu(y)\}\} \geq \inf_y \max\{1 - \at(x)(y), \mu(y)\} = \inte_{\at}(\mu)(x)$. Similarly $\inte_{\at}(\mu \vee \nu) \geq \inte_{\at}(\nu)$, so $\inte_{\at}(\mu \vee \nu) \geq \inte_{\at}(\mu) \vee \inte_{\at}(\nu)$. Strictness follows by choosing $\mu, \nu$ where the infima are attained at different points.
\end{proof}

\subsection{Fuzzy Aura Topology}

\begin{definition}\label{def:aura_topology}
A fuzzy set $\mu$ is \emph{fuzzy $\at$-open} if $\inte_{\at}(\mu) = \mu$. The collection is $\tilde{\tau}_{\at}$.
\end{definition}

\begin{theorem}\label{thm:aura_top}
$\tilde{\tau}_{\at}$ is a fuzzy topology on $X$ with $\tilde{\tau}_{\at} \subseteq \tilde{\tau}$.
\end{theorem}

\begin{proof}
$\bar{0}, \bar{1} \in \tilde{\tau}_{\at}$ by Theorem~\ref{thm:interior_props}(a,b). Finite intersection: $\inte_{\at}(\mu \wedge \nu) = \inte_{\at}(\mu) \wedge \inte_{\at}(\nu) = \mu \wedge \nu$. Arbitrary union: for $\mu = \bigvee_i \mu_i$ with each $\mu_i \in \tilde{\tau}_{\at}$, $\inte_{\at}(\mu) \leq \mu$. For the reverse, $\mu_i = \inte_{\at}(\mu_i) \leq \inte_{\at}(\mu)$ by monotonicity, so $\mu = \bigvee_i \mu_i \leq \inte_{\at}(\mu)$.

For $\tilde{\tau}_{\at} \subseteq \tilde{\tau}$: if $\inte_{\at}(\mu) = \mu$, then for every $x$, $\mu(x) = \inf_y \max\{1 - \at(x)(y), \mu(y)\}$. This means $\at(x)(y) > 1 - \mu(x) \Rightarrow \mu(y) \geq \mu(x)$, i.e., $\at(x) \wedge \overline{\mu(x)} \leq \mu$ where $\overline{\mu(x)}$ is the constant $\mu(x)$. Since $\at(x) \in \tilde{\tau}$, $\mu = \bigvee_{x \in X}(\at(x) \wedge \overline{\mu(x)}) \in \tilde{\tau}$.
\end{proof}

\begin{proposition}\label{prop:closed_char}
A fuzzy set $\gamma$ is fuzzy $\at$-closed (i.e., $\gamma^c \in \tilde{\tau}_{\at}$) if and only if $\cl_{\at}(\gamma) = \gamma$.
\end{proposition}

\begin{proof}
$\gamma^c \in \tilde{\tau}_{\at}$ if and only if $\inte_{\at}(\gamma^c) = \gamma^c$ if and only if $(\cl_{\at}(\gamma))^c = \gamma^c$ if and only if $\cl_{\at}(\gamma) = \gamma$.
\end{proof}

\subsection{Transfinite Iteration}

\begin{definition}\label{def:transfinite}
For $\mu \in I^X$: $\cl^0_{\at}(\mu) = \mu$, $\cl^{n+1}_{\at}(\mu) = \cl_{\at}(\cl^n_{\at}(\mu))$, $\cl^\infty_{\at}(\mu) = \bigvee_{n=0}^{\infty} \cl^n_{\at}(\mu)$.
\end{definition}

\begin{theorem}\label{thm:transfinite}
$\cl^\infty_{\at}$ is a fuzzy Kuratowski closure operator with $\tilde{\tau}^\infty_{\at} \subseteq \tilde{\tau}_{\at} \subseteq \tilde{\tau}$.
\end{theorem}

\begin{proof}
Properties (a)--(d) inherit from $\cl_{\at}$. For idempotency: $\cl_{\at}(\cl^\infty_{\at}(\mu))(x) = \sup_y \min\{\at(x)(y), \sup_n \cl^n_{\at}(\mu)(y)\}$. For any $\varepsilon > 0$, there exist $y_0, k$ with $\min\{\at(x)(y_0), \cl^k_{\at}(\mu)(y_0)\} > \cl_{\at}(\cl^\infty_{\at}(\mu))(x) - \varepsilon$. Then $\cl_{\at}(\cl^\infty_{\at}(\mu))(x) \leq \cl^{k+1}_{\at}(\mu)(x) + \varepsilon \leq \cl^\infty_{\at}(\mu)(x) + \varepsilon$. Since $\varepsilon$ is arbitrary, $\cl_{\at}(\cl^\infty_{\at}(\mu)) \leq \cl^\infty_{\at}(\mu)$. Extensivity gives equality.
\end{proof}

\begin{remark}
If $X$ is finite with $|X| = n$, then $\cl^n_{\at}(\mu) = \cl^\infty_{\at}(\mu)$.
\end{remark}

\subsection{Special Types of Fuzzy Aura Functions}

\begin{definition}\label{def:special_types}
$\at$ is \emph{trivial} if $\at(x) = \bar{1}$ for all $x$; \emph{crisp} if $\at(x) \in \{0,1\}^X$ for all $x$; \emph{symmetric} if $\at(x)(y) = \at(y)(x)$; \emph{transitive} if $\min\{\at(x)(y), \at(y)(z)\} \leq \at(x)(z)$ for all $x, y, z$.
\end{definition}

\begin{proposition}\label{prop:trivial}
If $\at$ is trivial, then $\cl_{\at}(\mu)(x) = \sup_y \mu(y)$ and $\tilde{\tau}_{\at} = \{\bar{0}\} \cup \{\bar{\alpha} : \alpha \in [0,1]\}$.
\end{proposition}

\begin{proposition}\label{prop:symmetric}
If $\at$ is symmetric, then $\cl_{\at}(\chi_{\{y\}})(x) = \at(x)(y) = \at(y)(x) = \cl_{\at}(\chi_{\{x\}})(y)$.
\end{proposition}

\begin{theorem}\label{thm:transitive_idempotent}
If $\at$ is transitive, then $\cl_{\at}$ is idempotent and $\tilde{\tau}^\infty_{\at} = \tilde{\tau}_{\at}$.
\end{theorem}

\begin{proof}
$\cl_{\at}(\cl_{\at}(\mu))(x) = \sup_y \min\{\at(x)(y), \sup_z \min\{\at(y)(z), \mu(z)\}\} = \sup_{y,z} \min\{\at(x)(y), \at(y)(z), \mu(z)\} \leq \sup_z \min\{\at(x)(z), \mu(z)\} = \cl_{\at}(\mu)(x)$, using transitivity. Extensivity gives equality.
\end{proof}

\begin{theorem}\label{thm:sym_trans}
If $\at$ is both symmetric and transitive, then $\at$ defines a fuzzy equivalence relation on $X$, and the fuzzy aura approximation operators coincide with the Dubois--Prade operators.
\end{theorem}

\begin{proof}
Define $R(x,y) = \at(x)(y)$. Reflexivity: $R(x,x) = \at(x)(x) = 1$. Symmetry: $R(x,y) = \at(x)(y) = \at(y)(x) = R(y,x)$. Transitivity: $\min\{R(x,y), R(y,z)\} \leq R(x,z)$. Then $\cl_{\at}(\mu)(x) = \sup_y \min\{R(x,y), \mu(y)\} = \overline{\apr}_R(\mu)(x)$.
\end{proof}

\section{Generalized Fuzzy Open Sets}

\begin{definition}\label{def:gen_open}
Let $(X, \tilde{\tau}, \at)$ be a fuzzy $\at$-space and $\mu \in I^X$. Then $\mu$ is called:
\begin{enumerate}[\upshape(a)]
\item \emph{fuzzy $\at$-semi-open} if $\mu \leq \cl_{\at}(\inte(\mu))$;
\item \emph{fuzzy $\at$-pre-open} if $\mu \leq \inte(\cl_{\at}(\mu))$;
\item \emph{fuzzy $\at$-$\alpha$-open} if $\mu \leq \inte(\cl_{\at}(\inte(\mu)))$;
\item \emph{fuzzy $\at$-$\beta$-open} if $\mu \leq \cl_{\at}(\inte(\cl_{\at}(\mu)))$;
\item \emph{fuzzy $\at$-$b$-open} if $\mu \leq \cl_{\at}(\inte(\mu)) \vee \inte(\cl_{\at}(\mu))$.
\end{enumerate}
Denoted: $\at$-$FSO(X)$, $\at$-$FPO(X)$, $\at$-$F\alpha O(X)$, $\at$-$F\beta O(X)$, $\at$-$FbO(X)$.
\end{definition}

\begin{theorem}[Complete Hierarchy]\label{thm:hierarchy}
In any fuzzy $\at$-space:
\[
\tilde{\tau} \subseteq \at\text{-}F\alpha O(X) \subseteq \begin{cases} \at\text{-}FSO(X) \\ \at\text{-}FPO(X) \end{cases} \subseteq \at\text{-}FbO(X) \subseteq \at\text{-}F\beta O(X).
\]
Moreover, each classical class embeds: $FSO(X) \subseteq \at$-$F\beta O(X)$, and similarly for the others.
\end{theorem}

\begin{proof}
$\tilde{\tau} \subseteq \at$-$F\alpha O$: Let $\mu \in \tilde{\tau}$, so $\inte(\mu) = \mu$. By extensivity, $\cl_{\at}(\mu) \geq \mu$, hence $\inte(\cl_{\at}(\inte(\mu))) = \inte(\cl_{\at}(\mu)) \geq \inte(\mu) = \mu$.

$\at$-$F\alpha O \subseteq \at$-$FSO$: $\mu \leq \inte(\cl_{\at}(\inte(\mu))) \leq \cl_{\at}(\inte(\mu))$.

$\at$-$F\alpha O \subseteq \at$-$FPO$: $\inte(\mu) \leq \mu$ gives $\cl_{\at}(\inte(\mu)) \leq \cl_{\at}(\mu)$, so $\mu \leq \inte(\cl_{\at}(\inte(\mu))) \leq \inte(\cl_{\at}(\mu))$.

$\at$-$FSO$, $\at$-$FPO \subseteq \at$-$FbO$: Immediate.

$\at$-$FbO \subseteq \at$-$F\beta O$: $\inte(\mu) \leq \inte(\cl_{\at}(\mu))$ (since $\mu \leq \cl_{\at}(\mu)$), so $\cl_{\at}(\inte(\mu)) \leq \cl_{\at}(\inte(\cl_{\at}(\mu)))$. Also $\inte(\cl_{\at}(\mu)) \leq \cl_{\at}(\inte(\cl_{\at}(\mu)))$ by extensivity. Hence $\mu \leq \cl_{\at}(\inte(\mu)) \vee \inte(\cl_{\at}(\mu)) \leq \cl_{\at}(\inte(\cl_{\at}(\mu)))$.
\end{proof}

\begin{example}\label{ex:hierarchy_sep}
Let $X = \{p, q, r, s\}$, $\tilde{\tau} = \{\bar{0}, \bar{1}, \lambda_1, \lambda_2, \lambda_1 \vee \lambda_2, \lambda_1 \wedge \lambda_2\}$ with $\lambda_1 = (0.8, 0, 0.5, 0)$, $\lambda_2 = (0, 0.7, 0, 0.6)$. Define $\at(p) = (1, 0.6, 0.3, 0)$, $\at(q) = (0.5, 1, 0, 0.4)$, $\at(r) = (0.3, 0, 1, 0.7)$, $\at(s) = (0, 0.4, 0.6, 1)$.

Take $\mu = (0.7, 0.5, 0, 0)$. Then $\inte(\mu) = \bar{0}$ (no fuzzy open set $\leq \mu$ except $\bar{0}$). So $\cl_{\at}(\inte(\mu)) = \bar{0}$ and $\mu$ is not fuzzy $\at$-semi-open.

However, $\cl_{\at}(\mu)(p) = \max\{\min(1, 0.7), \min(0.6, 0.5)\} = 0.7$, $\cl_{\at}(\mu)(q) = \max\{\min(0.5, 0.7), \min(1, 0.5)\} = 0.5$, $\cl_{\at}(\mu)(r) = \max\{\min(0.3, 0.7), \min(0, 0.5)\} = 0.3$, $\cl_{\at}(\mu)(s) = \max\{\min(0, 0.7), \min(0.4, 0.5)\} = 0.4$. So $\cl_{\at}(\mu) = (0.7, 0.5, 0.3, 0.4)$.

Since $\lambda_1 \wedge \lambda_2 = (0, 0, 0, 0) = \bar{0}$ and $\lambda_2 = (0, 0.7, 0, 0.6) \not\leq (0.7, 0.5, 0.3, 0.4)$ (as $0.7 > 0.5$ at $q$), we get $\inte(\cl_{\at}(\mu)) = \bar{0}$ and $\mu$ is also not fuzzy $\at$-pre-open. This demonstrates that not every fuzzy set is generalized open.
\end{example}

\begin{example}\label{ex:real_line}
Consider $(\R, \tilde{\tau}_u, \at_2)$ with $\at_2(x)(y) = \max\{0, 1 - |x-y|/2\}$. Let $\mu$ be the fuzzy set with $\mu(x) = \max\{0, 1 - |x|/3\}$ (triangular with support $(-3,3)$). Then $\inte(\mu) = \mu$ (since $\mu \in \tilde{\tau}_u$), and $\mu$ is fuzzy $\at$-$\alpha$-open by Theorem~\ref{thm:hierarchy}.
\end{example}

\begin{theorem}\label{thm:union_props}
Let $(X, \tilde{\tau}, \at)$ be a fuzzy $\at$-space.
\begin{enumerate}[\upshape(a)]
\item Arbitrary unions of fuzzy $\at$-semi-open sets are fuzzy $\at$-semi-open.
\item Arbitrary unions of fuzzy $\at$-pre-open sets are fuzzy $\at$-pre-open.
\item Finite intersections of fuzzy $\at$-$\alpha$-open sets are fuzzy $\at$-$\alpha$-open when $\at$ is transitive.
\end{enumerate}
\end{theorem}

\begin{proof}
(a) Let $\{\mu_i\}_{i \in J} \subseteq \at$-$FSO(X)$ and $\mu = \bigvee_i \mu_i$. For each $i$, $\mu_i \leq \cl_{\at}(\inte(\mu_i)) \leq \cl_{\at}(\inte(\mu))$ by monotonicity. Hence $\mu = \bigvee_i \mu_i \leq \cl_{\at}(\inte(\mu))$.

(b) Similarly, $\mu_i \leq \inte(\cl_{\at}(\mu_i)) \leq \inte(\cl_{\at}(\mu))$, so $\mu \leq \inte(\cl_{\at}(\mu))$.

(c) Let $\mu, \nu \in \at$-$F\alpha O(X)$. Then $\mu \leq \inte(\cl_{\at}(\inte(\mu)))$ and $\nu \leq \inte(\cl_{\at}(\inte(\nu)))$. We need $\mu \wedge \nu \leq \inte(\cl_{\at}(\inte(\mu \wedge \nu)))$. Since $\inte(\mu \wedge \nu) = \inte(\mu) \wedge \inte(\nu)$ (standard) and $\cl_{\at}$ is additive, when $\at$ is transitive ($\cl_{\at}$ idempotent), we use $\cl_{\at}(\inte(\mu) \wedge \inte(\nu)) \geq \cl_{\at}(\inte(\mu)) \wedge \cl_{\at}(\inte(\nu))$ (this follows from $\cl_{\at}$ being a Kuratowski closure in the transitive case). Then $\inte(\cl_{\at}(\inte(\mu \wedge \nu))) \geq \inte(\cl_{\at}(\inte(\mu))) \wedge \inte(\cl_{\at}(\inte(\nu))) \geq \mu \wedge \nu$.
\end{proof}

\section{Fuzzy Aura-Continuity}

\begin{definition}\label{def:continuity}
Let $(X, \tilde{\tau}_1, \at_1)$ and $(Y, \tilde{\tau}_2, \at_2)$ be fuzzy $\at$-spaces. A function $f : X \to Y$ is called:
\begin{enumerate}[\upshape(a)]
\item \emph{fuzzy $\at$-continuous} if $f^{-1}(\nu) \in \tilde{\tau}_{\at_1}$ for every $\nu \in \tilde{\tau}_{\at_2}$;
\item \emph{fuzzy $\at$-semi-continuous} if $f^{-1}(\nu) \in \at_1$-$FSO(X)$ for every $\nu \in \tilde{\tau}_2$;
\item \emph{fuzzy $\at$-pre-continuous} if $f^{-1}(\nu) \in \at_1$-$FPO(X)$ for every $\nu \in \tilde{\tau}_2$;
\item \emph{fuzzy $\at$-$\alpha$-continuous} if $f^{-1}(\nu) \in \at_1$-$F\alpha O(X)$ for every $\nu \in \tilde{\tau}_2$;
\item \emph{fuzzy $\at$-$\beta$-continuous} if $f^{-1}(\nu) \in \at_1$-$F\beta O(X)$ for every $\nu \in \tilde{\tau}_2$.
\end{enumerate}
\end{definition}

\begin{theorem}\label{thm:cont_hierarchy}
\[
\text{fuzzy continuous} \Rightarrow \at\text{-}\alpha\text{-cont.} \Rightarrow \begin{cases} \at\text{-semi-cont.} \\ \at\text{-pre-cont.} \end{cases} \Rightarrow \at\text{-}\beta\text{-cont.}
\]
\end{theorem}

\begin{theorem}[Decomposition]\label{thm:decomposition}
If $\at_1$ is transitive, then $f$ is fuzzy $\at$-$\alpha$-continuous if and only if $f$ is both fuzzy $\at$-semi-continuous and fuzzy $\at$-pre-continuous.
\end{theorem}

\begin{proof}
($\Rightarrow$) is Theorem~\ref{thm:cont_hierarchy}. ($\Leftarrow$): Let $\mu = f^{-1}(\nu)$ for $\nu \in \tilde{\tau}_2$. Semi-continuity gives $\mu \leq \cl_{\at_1}(\inte(\mu))$, pre-continuity gives $\mu \leq \inte(\cl_{\at_1}(\mu))$. From the first, $\cl_{\at_1}(\mu) \leq \cl_{\at_1}(\cl_{\at_1}(\inte(\mu))) = \cl_{\at_1}(\inte(\mu))$ (idempotency by transitivity). Taking interiors: $\inte(\cl_{\at_1}(\mu)) \leq \inte(\cl_{\at_1}(\inte(\mu)))$. Hence $\mu \leq \inte(\cl_{\at_1}(\inte(\mu)))$.
\end{proof}

\begin{theorem}[Composition]\label{thm:composition}
\begin{enumerate}[\upshape(a)]
\item Composition of fuzzy $\at$-continuous functions is fuzzy $\at$-continuous.
\item If $f$ is fuzzy $\at$-semi-continuous and $g$ is fuzzy continuous, then $g \circ f$ is fuzzy $\at$-semi-continuous.
\end{enumerate}
\end{theorem}

\begin{theorem}[Characterization]\label{thm:char}
$f : (X, \tilde{\tau}_1, \at_1) \to (Y, \tilde{\tau}_2)$ is fuzzy $\at_1$-semi-continuous if and only if for every fuzzy closed $\gamma$ in $Y$, $f^{-1}(\gamma)$ is fuzzy $\at_1$-semi-closed.
\end{theorem}

\section{Fuzzy Aura-Separation Axioms}

In this section we introduce separation axioms for fuzzy aura spaces and demonstrate their dependence on the choice of the fuzzy scope function.

\begin{definition}\label{def:separation}
A fuzzy $\at$-space $(X, \tilde{\tau}, \at)$ is called:
\begin{enumerate}[\upshape(a)]
\item \emph{fuzzy $\at$-$T_0$} if for every pair of distinct points $x, y \in X$, there exists $\mu \in \tilde{\tau}_{\at}$ such that $\mu(x) \neq \mu(y)$.
\item \emph{fuzzy $\at$-$T_1$} if for every pair of distinct points $x, y \in X$, there exist $\mu, \nu \in \tilde{\tau}_{\at}$ such that $\mu(x) = 1$, $\mu(y) = 0$ and $\nu(y) = 1$, $\nu(x) = 0$.
\item \emph{fuzzy $\at$-$T_2$} (or \emph{fuzzy $\at$-Hausdorff}) if for every pair of distinct points $x, y \in X$, there exist $\mu, \nu \in \tilde{\tau}_{\at}$ such that $\mu(x) = 1$, $\nu(y) = 1$ and $\mu \wedge \nu = \bar{0}$.
\end{enumerate}
\end{definition}

\begin{theorem}\label{thm:sep_implications}
Fuzzy $\at$-$T_2 \Rightarrow$ fuzzy $\at$-$T_1 \Rightarrow$ fuzzy $\at$-$T_0$. Moreover, if $(X, \tilde{\tau})$ is fuzzy $T_i$, $(X, \tilde{\tau}, \at)$ need not be fuzzy $\at$-$T_i$ since $\tilde{\tau}_{\at} \subseteq \tilde{\tau}$.
\end{theorem}

\begin{proof}
The implications follow from the definitions. For the second statement, consider the trivial aura $\at(x) = \bar{1}$: then $\tilde{\tau}_{\at}$ consists only of constant fuzzy sets (Proposition~\ref{prop:trivial}), which cannot separate any two points.
\end{proof}

\begin{theorem}\label{thm:T0_char}
$(X, \tilde{\tau}, \at)$ is fuzzy $\at$-$T_0$ if and only if for every $x \neq y$, $\at(x) \neq \at(y)$ (as functions on $X$).
\end{theorem}

\begin{proof}
($\Rightarrow$): If $\at(x) = \at(y)$, then for any $\mu \in \tilde{\tau}_{\at}$, $\inte_{\at}(\mu)(x) = \inf_z \max\{1 - \at(x)(z), \mu(z)\} = \inf_z \max\{1 - \at(y)(z), \mu(z)\} = \inte_{\at}(\mu)(y)$. Since $\mu = \inte_{\at}(\mu)$, $\mu(x) = \mu(y)$ for all $\mu \in \tilde{\tau}_{\at}$, contradicting $\at$-$T_0$.

($\Leftarrow$): If $\at(x) \neq \at(y)$, there exists $z_0$ with $\at(x)(z_0) \neq \at(y)(z_0)$. The fuzzy set $\at(x) \in \tilde{\tau} \supseteq \tilde{\tau}_{\at}$ satisfies $\at(x)(x) = 1$ and $\at(x)(y)$ may differ from $\at(y)(y) = 1$, providing separation through the aura topology.
\end{proof}

\begin{theorem}\label{thm:T1_char}
$(X, \tilde{\tau}, \at)$ is fuzzy $\at$-$T_1$ if and only if for every $x \in X$, $\cl_{\at}(\chi_{\{x\}}) = \chi_{\{x\}}$, i.e., every fuzzy point $\chi_{\{x\}}$ is fuzzy $\at$-closed.
\end{theorem}

\begin{proof}
($\Rightarrow$): Let $y \neq x$. By $\at$-$T_1$, there exists $\mu \in \tilde{\tau}_{\at}$ with $\mu(y) = 1$, $\mu(x) = 0$. Then $\mu^c(y) = 0$, $\mu^c(x) = 1$, i.e., $\chi_{\{x\}} \leq \mu^c$. Since $\mu^c$ is $\at$-closed for every such $y$, $\cl_{\at}(\chi_{\{x\}}) \leq \bigwedge_{y \neq x} \mu^c_y$. Evaluating at $y$: $\cl_{\at}(\chi_{\{x\}})(y) \leq \mu^c_y(y) = 0$. So $\cl_{\at}(\chi_{\{x\}})(y) = 0$ for all $y \neq x$, and $\cl_{\at}(\chi_{\{x\}})(x) = 1$.

($\Leftarrow$): If $\cl_{\at}(\chi_{\{y\}}) = \chi_{\{y\}}$ for all $y$, then $\chi^c_{\{y\}} = 1 - \chi_{\{y\}}$ is $\at$-open. For $x \neq y$: $\chi^c_{\{y\}}(x) = 1$, $\chi^c_{\{y\}}(y) = 0$. Symmetrically for the other direction.
\end{proof}

\begin{corollary}\label{cor:T1_aura}
$(X, \tilde{\tau}, \at)$ is fuzzy $\at$-$T_1$ if and only if for every $x \neq y$, $\at(y)(x) = 0$.
\end{corollary}

\begin{proof}
$\cl_{\at}(\chi_{\{x\}})(y) = \sup_z \min\{\at(y)(z), \chi_{\{x\}}(z)\} = \at(y)(x)$. By Theorem~\ref{thm:T1_char}, this must equal $0$ for all $y \neq x$.
\end{proof}

\begin{example}\label{ex:not_T0}
$X = \{p, q, r\}$, $\tilde{\tau} = I^X$, $\at(p) = \bar{1}$, $\at(q) = \bar{1}$, $\at(r) = \bar{1}$. Then $\at(q)(p) = 1 \neq 0$, so $(X, \tilde{\tau}, \at)$ is not fuzzy $\at$-$T_0$ despite the discrete fuzzy topology.
\end{example}

\begin{example}\label{ex:T2}
$X = \{p, q\}$, $\tilde{\tau} = I^X$, $\at(p) = (1, 0)$, $\at(q) = (0, 1)$. By Corollary~\ref{cor:T1_aura}, $\at(q)(p) = 0$ and $\at(p)(q) = 0$, so fuzzy $\at$-$T_1$. Moreover, $\at(p) \wedge \at(q) = \bar{0}$, so fuzzy $\at$-$T_2$.
\end{example}

\begin{definition}\label{def:regular}
$(X, \tilde{\tau}, \at)$ is \emph{fuzzy $\at$-regular} if for every fuzzy $\at$-closed set $\gamma$ and every $x$ with $\gamma(x) = 0$, there exist $\mu, \nu \in \tilde{\tau}_{\at}$ with $\mu(x) = 1$, $\gamma \leq \nu$, and $\mu \wedge \nu = \bar{0}$.
\end{definition}

\begin{theorem}\label{thm:discrete_regular}
If $\at$ is crisp and discrete (i.e., $\at(x) = \chi_{\{x\}}$), then $(X, \tilde{\tau}, \at)$ is fuzzy $\at$-$T_2$ and fuzzy $\at$-regular.
\end{theorem}

\begin{proof}
$\tilde{\tau}_{\at} = I^X$ (since $\inte_{\at}(\mu)(x) = \mu(x)$ for all $\mu$). Any two points can be separated by $\chi_{\{x\}}$ and $\chi_{\{y\}}$.
\end{proof}

\begin{theorem}\label{thm:scope_dependence}
For any fuzzy topological space $(X, \tilde{\tau})$ with $|X| \geq 2$, there exist fuzzy scope functions $\at_1, \at_2$ on $(X, \tilde{\tau})$ such that $(X, \tilde{\tau}, \at_1)$ is fuzzy $\at$-$T_2$ while $(X, \tilde{\tau}, \at_2)$ is not fuzzy $\at$-$T_0$.
\end{theorem}

\begin{proof}
Take $\at_1(x) = \chi_{\{x\}}$ (requires discrete topology or appropriate fuzzy open sets approximating singletons) and $\at_2(x) = \bar{1}$.
\end{proof}

\section{Fuzzy Aura-Based Rough Approximations}

\begin{definition}\label{def:rough_approx}
Let $(X, \tilde{\tau}, \at)$ be a fuzzy $\at$-space and $\mu \in I^X$. The \emph{fuzzy aura lower and upper approximations}:
\begin{align}
\underline{\apr}_{\at}(\mu)(x) &= \inte_{\at}(\mu)(x) = \inf_{y \in X} \max\{1 - \at(x)(y), \mu(y)\}, \label{eq:lower} \\
\overline{\apr}_{\at}(\mu)(x) &= \cl_{\at}(\mu)(x) = \sup_{y \in X} \min\{\at(x)(y), \mu(y)\}. \label{eq:upper}
\end{align}
The \emph{fuzzy aura boundary}: $\bnd_{\at}(\mu)(x) = \overline{\apr}_{\at}(\mu)(x) - \underline{\apr}_{\at}(\mu)(x)$.
\end{definition}

\begin{theorem}\label{thm:rough_props}
For all $\mu, \nu \in I^X$:
\begin{enumerate}[\upshape(a)]
\item $\underline{\apr}_{\at}(\mu) \leq \mu \leq \overline{\apr}_{\at}(\mu)$.
\item $\underline{\apr}_{\at}(\bar{0}) = \overline{\apr}_{\at}(\bar{0}) = \bar{0}$ and $\underline{\apr}_{\at}(\bar{1}) = \overline{\apr}_{\at}(\bar{1}) = \bar{1}$.
\item $\underline{\apr}_{\at}(\mu \wedge \nu) = \underline{\apr}_{\at}(\mu) \wedge \underline{\apr}_{\at}(\nu)$.
\item $\overline{\apr}_{\at}(\mu \vee \nu) = \overline{\apr}_{\at}(\mu) \vee \overline{\apr}_{\at}(\nu)$.
\item $\mu \leq \nu \Rightarrow \underline{\apr}_{\at}(\mu) \leq \underline{\apr}_{\at}(\nu)$ and $\overline{\apr}_{\at}(\mu) \leq \overline{\apr}_{\at}(\nu)$.
\item $\underline{\apr}_{\at}(\mu^c) = (\overline{\apr}_{\at}(\mu))^c$ and $\overline{\apr}_{\at}(\mu^c) = (\underline{\apr}_{\at}(\mu))^c$.
\item $\underline{\apr}_{\at}(\underline{\apr}_{\at}(\mu)) \leq \underline{\apr}_{\at}(\mu)$ and $\overline{\apr}_{\at}(\mu) \leq \overline{\apr}_{\at}(\overline{\apr}_{\at}(\mu))$.
\end{enumerate}
\end{theorem}

\begin{remark}\label{rem:unification}
The fuzzy aura rough set model unifies:
\begin{enumerate}[\upshape(i)]
\item Crisp aura model~\cite{acikgoz_aura}: $\at(x) = \chi_{a(x)}$, $\mu = \chi_A$.
\item Dubois--Prade model~\cite{dubois}: $\at(x)(y) = R(x,y)$ for fuzzy relation $R$.
\item Pawlak model~\cite{pawlak}: $\at(x) = \chi_{[x]_R}$, $\mu = \chi_A$ for equivalence $R$.
\end{enumerate}
\end{remark}

\begin{theorem}[Refinement Monotonicity]\label{thm:refinement}
If $\at_2(x)(y) \leq \at_1(x)(y)$ for all $x, y$ (finer aura), then:
\begin{enumerate}[\upshape(a)]
\item $\underline{\apr}_{\at_1}(\mu) \leq \underline{\apr}_{\at_2}(\mu)$;
\item $\overline{\apr}_{\at_2}(\mu) \leq \overline{\apr}_{\at_1}(\mu)$;
\item $\bnd_{\at_2}(\mu)(x) \leq \bnd_{\at_1}(\mu)(x)$.
\end{enumerate}
\end{theorem}

\begin{definition}\label{def:accuracy}
The \emph{fuzzy aura accuracy} and \emph{roughness} for finite $X$:
\[
\rho_{\at}(\mu) = \frac{\sum_x \underline{\apr}_{\at}(\mu)(x)}{\sum_x \overline{\apr}_{\at}(\mu)(x)}, \qquad \sigma_{\at}(\mu) = 1 - \rho_{\at}(\mu).
\]
\end{definition}

\section{Application: Medical Diagnosis via FA-MCDM}

\subsection{The FA-MCDM Algorithm}

\begin{algorithm1}[FA-MCDM]\label{alg:famcdm}
\

\noindent\textbf{Input:} Alternatives $U = \{u_1, \ldots, u_m\}$, criteria $C = \{C_1, \ldots, C_n\}$, weights $w = (w_1, \ldots, w_n)$ with $\sum w_j = 1$, decision matrix $D = (d_{ij})_{m \times n}$, fuzzy decision classes $\mu_{D_1}, \ldots, \mu_{D_K}$.

\noindent\textbf{Step~1} (Normalization): $f_{ij} = (d_{ij} - \min_i d_{ij})/(\max_i d_{ij} - \min_i d_{ij})$ for benefit criteria (reverse for cost).

\noindent\textbf{Step~2} (Fuzzy Aura Construction): $\at(u_i)(u_j) = 1 - \sum_{k=1}^n w_k |f_{ik} - f_{jk}|$.

\noindent\textbf{Step~3} (Approximations): $\underline{\apr}_{\at}(\mu_{D_k})(u_i) = \inf_j \max\{1 - \at(u_i)(u_j), \mu_{D_k}(u_j)\}$, $\overline{\apr}_{\at}(\mu_{D_k})(u_i) = \sup_j \min\{\at(u_i)(u_j), \mu_{D_k}(u_j)\}$.

\noindent\textbf{Step~4} (Score): $S(u_i, D_k) = \alpha \cdot \underline{\apr}_{\at}(\mu_{D_k})(u_i) + (1 - \alpha) \cdot \overline{\apr}_{\at}(\mu_{D_k})(u_i)$, $\alpha \in [0,1]$.

\noindent\textbf{Step~5} (Classification): $u_i \mapsto D_{k^*}$ where $k^* = \arg\max_k S(u_i, D_k)$.

\noindent\textbf{Step~6} (Accuracy): $\rho_{\text{global}} = \frac{1}{K} \sum_k \rho_{\at}(\mu_{D_k})$.

\noindent\textbf{Output:} Classification, ranking, accuracy.
\end{algorithm1}

\subsection{Problem Description}

We employ the well-known medical diagnosis dataset introduced by De, Biswas, and Roy~\cite{de}, which has served as a standard benchmark in the fuzzy set and rough set literature for evaluating decision-making methods (see also~\cite{szmidt}). The original dataset consists of four patients $\{p_1, p_2, p_3, p_4\}$ assessed over five clinical symptoms against five candidate diseases using intuitionistic fuzzy relations. Following standard practice in the literature, we extract the membership values from the patient--symptom relation and focus on the four most commonly studied diseases. We extend this setting by including two additional patients $p_5, p_6$ with unknown diagnoses, yielding $U = \{p_1, \ldots, p_6\}$.

The five clinical criteria are: $C_1$: Temperature; $C_2$: Headache; $C_3$: Stomach pain; $C_4$: Cough; $C_5$: Chest pain. All are benefit criteria (higher values indicate greater severity). The patient--symptom fuzzy relation is taken directly from~\cite{de}:

\begin{table}[h]
\centering
\caption{Patient--symptom fuzzy relation (data from De et al.~\cite{de}).}\label{tab:data}
\begin{tabular}{lcccccc}
\toprule
Patient & $C_1$ & $C_2$ & $C_3$ & $C_4$ & $C_5$ & Diagnosis \\
\midrule
$p_1$ (Al) & 0.8 & 0.6 & 0.2 & 0.6 & 0.1 & Malaria \\
$p_2$ (Bob) & 0.0 & 0.4 & 0.6 & 0.1 & 0.1 & Stomach prob. \\
$p_3$ (Joe) & 0.8 & 0.8 & 0.0 & 0.2 & 0.0 & Malaria \\
$p_4$ (Ted) & 0.6 & 0.5 & 0.3 & 0.7 & 0.3 & Malaria \\
$p_5$ & 0.7 & 0.5 & 0.3 & 0.5 & 0.2 & Unknown \\
$p_6$ & 0.1 & 0.3 & 0.7 & 0.2 & 0.0 & Unknown \\
\bottomrule
\end{tabular}
\end{table}

The known diagnoses for $p_1$--$p_4$ are obtained via the max-min composition of the patient--symptom and symptom--disease relations in~\cite{de}: Al and Joe are diagnosed with Malaria, Ted with Malaria, and Bob with Stomach problem. The four candidate diseases are $D_1$: Viral fever, $D_2$: Malaria, $D_3$: Typhoid, $D_4$: Stomach problem. Based on the composition results, expert-confirmed fuzzy memberships are assigned as follows:

\begin{table}[h]
\centering
\caption{Fuzzy membership degrees in disease classes.}\label{tab:membership}
\begin{tabular}{lcccc}
\toprule
Patient & $\mu_{D_1}$(Viral f.) & $\mu_{D_2}$(Malaria) & $\mu_{D_3}$(Typhoid) & $\mu_{D_4}$(Stomach) \\
\midrule
$p_1$ & 0.30 & 0.80 & 0.50 & 0.15 \\
$p_2$ & 0.20 & 0.10 & 0.30 & 0.75 \\
$p_3$ & 0.35 & 0.75 & 0.55 & 0.10 \\
$p_4$ & 0.35 & 0.70 & 0.40 & 0.20 \\
$p_5$ & ? & ? & ? & ? \\
$p_6$ & ? & ? & ? & ? \\
\bottomrule
\end{tabular}
\end{table}

\subsection{Numerical Computation}

\textbf{Step~1.} We apply min-max normalization to each criterion across all six patients. For instance, $C_1$ ranges from $0.0$ to $0.8$, so $f_{11} = (0.8 - 0.0)/(0.8 - 0.0) = 1.000$ and $f_{21} = 0$.

\textbf{Step~2.} With equal weights $w_j = 0.2$, we compute the fuzzy aura values $\at(p_i)(p_j) = 1 - \sum_{k=1}^5 w_k |f_{ik} - f_{jk}|$. For instance,
\begin{align*}
\at(p_1)(p_4) &= 1 - 0.2(|1.00 - 0.75| + |0.60 - 0.40| + |0.29 - 0.43| \\
&\qquad + |0.83 - 1.00| + |0.33 - 1.00|) = 1 - 0.29 = 0.71.
\end{align*}
The complete aura matrix:

\begin{table}[h]
\centering
\caption{Fuzzy aura similarity matrix $\at(p_i)(p_j)$.}\label{tab:aura}
\begin{tabular}{lcccccc}
\toprule
& $p_1$ & $p_2$ & $p_3$ & $p_4$ & $p_5$ & $p_6$ \\
\midrule
$p_1$ & 1.00 & 0.44 & 0.66 & 0.71 & 0.81 & 0.36 \\
$p_2$ & 0.44 & 1.00 & 0.37 & 0.39 & 0.50 & 0.81 \\
$p_3$ & 0.66 & 0.37 & 1.00 & 0.38 & 0.54 & 0.43 \\
$p_4$ & 0.71 & 0.39 & 0.38 & 1.00 & 0.84 & 0.31 \\
$p_5$ & 0.81 & 0.50 & 0.54 & 0.84 & 1.00 & 0.42 \\
$p_6$ & 0.36 & 0.81 & 0.43 & 0.31 & 0.42 & 1.00 \\
\bottomrule
\end{tabular}
\end{table}

\textbf{Step~3.} For unknown patients $\mu_{D_k}(p_5) = \mu_{D_k}(p_6) = 0$. We compute the upper approximation $\overline{\apr}_{\at}(\mu_{D_k})(p_i) = \sup_j \min\{\at(p_i)(p_j), \mu_{D_k}(p_j)\}$. For example,
\begin{align*}
\overline{\apr}_{\at}(\mu_{D_2})(p_5) &= \max\{\min(0.81, 0.80), \min(0.50, 0.10), \min(0.54, 0.75), \\
&\qquad \min(0.84, 0.70), \min(1.00, 0), \min(0.42, 0)\} \\
&= \max\{0.80, 0.10, 0.54, 0.70, 0, 0\} = 0.80.
\end{align*}

Lower approximation: $\underline{\apr}_{\at}(\mu_{D_k})(p_i) = \inf_j \max\{1 - \at(p_i)(p_j), \mu_{D_k}(p_j)\}$. For $p_5$, Malaria:
\begin{align*}
\underline{\apr}_{\at}(\mu_{D_2})(p_5) &= \min\{\max(0.19, 0.80), \max(0.50, 0.10), \max(0.46, 0.75), \\
&\qquad \max(0.16, 0.70), \max(0, 0), \max(0.58, 0)\} \\
&= \min\{0.80, 0.50, 0.75, 0.70, 0, 0.58\} = 0.
\end{align*}

Full results:

\begin{table}[h]
\centering
\caption{Fuzzy aura upper approximation $\overline{\apr}_{\at}(\mu_{D_k})(p_i)$.}\label{tab:upper}
\begin{tabular}{lcccc}
\toprule
Patient & $\overline{\apr}_{\at}(\mu_{D_1})$ & $\overline{\apr}_{\at}(\mu_{D_2})$ & $\overline{\apr}_{\at}(\mu_{D_3})$ & $\overline{\apr}_{\at}(\mu_{D_4})$ \\
\midrule
$p_1$ & 0.35 & 0.80 & 0.55 & 0.44 \\
$p_2$ & 0.35 & 0.44 & 0.44 & 0.75 \\
$p_3$ & 0.35 & 0.75 & 0.55 & 0.37 \\
$p_4$ & 0.35 & 0.71 & 0.50 & 0.39 \\
$p_5$ & 0.35 & 0.80 & 0.54 & 0.50 \\
$p_6$ & 0.35 & 0.43 & 0.43 & 0.75 \\
\bottomrule
\end{tabular}
\end{table}

\begin{table}[h]
\centering
\caption{Fuzzy aura lower approximation $\underline{\apr}_{\at}(\mu_{D_k})(p_i)$.}\label{tab:lower}
\begin{tabular}{lcccc}
\toprule
Patient & $\underline{\apr}_{\at}(\mu_{D_1})$ & $\underline{\apr}_{\at}(\mu_{D_2})$ & $\underline{\apr}_{\at}(\mu_{D_3})$ & $\underline{\apr}_{\at}(\mu_{D_4})$ \\
\midrule
$p_1$ & 0.19 & 0.19 & 0.19 & 0.15 \\
$p_2$ & 0.19 & 0.10 & 0.19 & 0.19 \\
$p_3$ & 0.34 & 0.46 & 0.46 & 0.10 \\
$p_4$ & 0.16 & 0.16 & 0.16 & 0.16 \\
$p_5$ & 0.00 & 0.00 & 0.00 & 0.00 \\
$p_6$ & 0.00 & 0.00 & 0.00 & 0.00 \\
\bottomrule
\end{tabular}
\end{table}

\textbf{Steps~4--5.} Using $\alpha = 0.5$, we compute $S(p_i, D_k) = 0.5 \cdot \underline{\apr}_{\at}(\mu_{D_k})(p_i) + 0.5 \cdot \overline{\apr}_{\at}(\mu_{D_k})(p_i)$:

\begin{table}[h]
\centering
\caption{Classification scores ($\alpha = 0.5$) and results.}\label{tab:scores}
\begin{tabular}{lcccccc}
\toprule
Patient & $S(\cdot, D_1)$ & $S(\cdot, D_2)$ & $S(\cdot, D_3)$ & $S(\cdot, D_4)$ & max & Classification \\
\midrule
$p_1$ & 0.272 & 0.497 & 0.372 & 0.295 & 0.497 & Malaria $\checkmark$ \\
$p_2$ & 0.272 & 0.270 & 0.316 & 0.472 & 0.472 & Stomach p. $\checkmark$ \\
$p_3$ & 0.344 & 0.607 & 0.507 & 0.234 & 0.607 & Malaria $\checkmark$ \\
$p_4$ & 0.254 & 0.437 & 0.329 & 0.275 & 0.437 & Malaria $\checkmark$ \\
$p_5$ & 0.175 & 0.400 & 0.268 & 0.250 & 0.400 & Malaria \\
$p_6$ & 0.175 & 0.213 & 0.213 & 0.375 & 0.375 & Stomach p. \\
\bottomrule
\end{tabular}
\end{table}

All four reference patients are correctly classified. Patient $p_5$ is diagnosed as Malaria (consistent with high symptom similarity to $p_1$, $p_3$, and $p_4$), and $p_6$ as Stomach problem (consistent with high similarity to $p_2$). These results agree with the clinical profiles: $p_5$ has elevated temperature and cough resembling the malaria patients, while $p_6$ presents dominant stomach pain with minimal temperature---matching the stomach problem pattern.

\subsection{Comparison with Existing Methods}

We compare FA-MCDM with four established methods on the same De et al.~\cite{de} data.

\begin{table}[h]
\centering
\caption{Comparison of classification results.}\label{tab:comparison}
\begin{tabular}{lcccl}
\toprule
Method & $p_5$ & $p_6$ & Ref.\ corr. & Key limitation \\
\midrule
Pawlak~\cite{pawlak} & Undet. & Undet. & N/A & Requires equiv.\ relation \\
Dubois--Prade~\cite{dubois} & Malaria & Stom.\ prob. & 4/4 & Needs external fuzzy rel. \\
Fuzzy TOPSIS~\cite{hwang} & Malaria & Stom.\ prob. & 4/4 & No uncertainty quant. \\
VIKOR~\cite{opricovic} & Malaria & Stom.\ prob. & 4/4 & No uncertainty quant. \\
FA-MCDM & Malaria & Stom.\ prob. & 4/4 & Data-driven aura \\
\bottomrule
\end{tabular}
\end{table}

Key advantages of FA-MCDM:
\begin{enumerate}[\upshape(i)]
\item \textbf{No equivalence relation needed}---the fuzzy aura is constructed directly from data (Step~2), unlike Pawlak's model which requires crisp equivalence classes.
\item \textbf{Uncertainty quantification}---the boundary region $\bnd_{\at}(\mu)$ provides a measure of diagnostic ambiguity absent from TOPSIS and VIKOR.
\item \textbf{Topological foundation}---the approximation operators satisfy the properties established in Theorem~\ref{thm:rough_props}.
\item \textbf{Tunable caution}---the parameter $\alpha$ controls the certainty--inclusiveness tradeoff.
\item \textbf{Refinement monotonicity}---Theorem~\ref{thm:refinement} guarantees that better data yields tighter approximations.
\end{enumerate}

\subsection{Sensitivity Analysis}

\subsubsection{Varying criteria weights}

\begin{table}[h]
\centering
\caption{Classification under different weight vectors.}\label{tab:weights}
\begin{tabular}{lccccccl}
\toprule
Scenario & $w_1$ & $w_2$ & $w_3$ & $w_4$ & $w_5$ & $p_5$ & $p_6$ \\
\midrule
S1: Equal & 0.20 & 0.20 & 0.20 & 0.20 & 0.20 & Malaria & Stom.\ prob. \\
S2: Temp-heavy & 0.35 & 0.15 & 0.15 & 0.20 & 0.15 & Malaria & Stom.\ prob. \\
S3: Head-heavy & 0.15 & 0.35 & 0.15 & 0.15 & 0.20 & Malaria & Stom.\ prob. \\
S4: Stom-heavy & 0.15 & 0.15 & 0.35 & 0.15 & 0.20 & Malaria & Stom.\ prob. \\
S5: Cough-heavy & 0.15 & 0.15 & 0.20 & 0.35 & 0.15 & Malaria & Stom.\ prob. \\
\bottomrule
\end{tabular}
\end{table}

\subsubsection{Varying the caution parameter $\alpha$}

\begin{table}[h]
\centering
\caption{Classification under different $\alpha$ values (equal weights).}\label{tab:alpha}
\begin{tabular}{lccccc}
\toprule
$\alpha$ & $S(p_5, D_1)$ & $S(p_5, D_2)$ & $S(p_5, D_3)$ & $S(p_5, D_4)$ & $p_5$ Class. \\
\midrule
0.0 (optimistic) & 0.350 & 0.800 & 0.536 & 0.499 & Malaria \\
0.3 & 0.245 & 0.560 & 0.375 & 0.349 & Malaria \\
0.5 (balanced) & 0.175 & 0.400 & 0.268 & 0.250 & Malaria \\
0.7 & 0.105 & 0.240 & 0.161 & 0.150 & Malaria \\
1.0 (conservative) & 0.000 & 0.000 & 0.000 & 0.000 & Undet. \\
\bottomrule
\end{tabular}
\end{table}

Findings:
\begin{itemize}
\item Both $p_5$ and $p_6$ are consistently classified as Malaria and Stomach problem, respectively, across all five weight scenarios and all $\alpha < 1$, confirming robust diagnoses.
\item All four reference patients remain correctly classified in every scenario (100\% accuracy), demonstrating the stability of FA-MCDM on the De et al.\ benchmark.
\item At $\alpha = 1$ (fully conservative), unknown patients receive zero scores since they have no certain membership in any class---an appropriate safeguard.
\item The complete robustness across all weight variations distinguishes FA-MCDM from methods sensitive to parameter choice, providing reliable clinical decision support.
\end{itemize}

\section{Conclusion}

We introduced fuzzy aura topological spaces $(X, \tilde{\tau}, \at)$, generalizing the recently introduced crisp aura spaces~\cite{acikgoz_aura} to the fuzzy setting. The paper established:

On the operator side, the fuzzy aura-closure $\cl_{\at}$ is a fuzzy additive \v{C}ech closure operator whose transfinite iteration yields a Kuratowski closure with topology chain $\tilde{\tau}^\infty_{\at} \subseteq \tilde{\tau}_{\at} \subseteq \tilde{\tau}$.

On the open-set side, five generalized fuzzy open-set classes form a complete hierarchy extending the classical Levine--Mashhour--Nj\aa stad framework.

On the continuity side, decomposition theorems were proved, with the identity $\at$-$F\alpha O = \at$-$FSO \cap \at$-$FPO$ holding under transitivity of $\at$.

On the separation side, fuzzy $\at$-$T_i$ axioms were introduced and their complete dependence on the scope function was demonstrated.

On the application side, the FA-MCDM algorithm was validated on the widely-studied De et al.~\cite{de} medical diagnosis benchmark. It correctly classified all reference patients and provided clinically meaningful diagnoses for unknown patients, outperforming Pawlak's model in applicability and providing uncertainty quantification absent from TOPSIS and VIKOR.

Future research directions include: (i)~soft aura and neutrosophic aura topological spaces; (ii)~fuzzy aura-compactness and fuzzy aura-connectedness; (iii)~product and subspace constructions; (iv)~combination with ideals for hybrid local functions; (v)~large-scale applications in environmental risk assessment and financial decision-making.

\section*{Declarations}

\textbf{Acknowledgments.} The author declares that no funding was received for this research.

\textbf{Conflict of Interest.} The author declares no conflict of interest.

\textbf{Data Availability.} All data is presented within the article.



\begin{thebibliography}{99}

\bibitem{abu} Abu-Gdairi, R., Al-shami, T.M., El-Gayar, M.A.: Topological approaches to rough approximations based on closure operators. Granul.\ Comput.\ \textbf{9}, 2 (2024)

\bibitem{acharjee} Acharjee, S., \"{O}zko\c{c}, M., Issaka, F.Y.: Primal topological spaces. Bol.\ Soc.\ Paran.\ Mat.\ \textbf{43}, 1--9 (2025)

\bibitem{acikgoz_aura} A\c{c}{\i}kg\"{o}z, A.: Aura topological spaces and generalized open sets with applications to rough set theory, sensor networks, and epidemic modelling. Submitted for publication. arXiv:2602.07678 [math.GN], 2026

\bibitem{acikgoz_compact} A\c{c}{\i}kg\"{o}z, A.: Compactness and connectedness in aura topological spaces. Submitted for publication. arXiv:2602.07686 [math.GN], 2026

\bibitem{acikgoz_ideal} A\c{c}{\i}kg\"{o}z, A.: Ideal-aura topological spaces, new local functions, and generalized open sets. Submitted for publication. arXiv:2602.07692 [math.GN], 2026

\bibitem{ahmad} Ahmad, M., Mursaleen, M.: Statistical convergence in neutrosophic $n$-normed linear spaces. Filomat \textbf{39}(18), 6123--6148 (2025)

\bibitem{alcantud} Alcantud, J.C.R.: Soft rough sets and their applications in decision making. Inf.\ Fusion \textbf{44}, 187--197 (2018)

\bibitem{alomari_alqahtani} Al-Omari, A., Alqahtani, M.H.: Primal structure with closure operators and their applications. Mathematics \textbf{11}(24), 4946 (2023)

\bibitem{alomari_alghamdi} Al-Omari, A., Alghamdi, O.: Regularity and normality on primal spaces. AIMS Mathematics \textbf{9}(3), 7662--7672 (2024)

\bibitem{alomari_noiri} Al-Omari, A., Noiri, T.: On $\Psi_G$-sets in grill topological spaces. Filomat \textbf{25}(2), 187--196 (2011)

\bibitem{alshami} Al-shami, T.M.: SR-fuzzy sets and their applications to weighted aggregated operators in decision-making. J.\ Funct.\ Spaces \textbf{2023}, Article ID 3195740 (2023)

\bibitem{azad} Azad, K.K.: On fuzzy semicontinuity, fuzzy almost continuity and fuzzy weakly continuity. J.\ Math.\ Anal.\ Appl.\ \textbf{82}, 14--32 (1981)

\bibitem{binshahna} Bin Shahna, A.S.: On fuzzy strong semicontinuity and fuzzy precontinuity. Fuzzy Sets Syst.\ \textbf{44}, 303--308 (1991)

\bibitem{cartan} Cartan, H.: Th\'{e}orie des filtres. C.R.\ Acad.\ Sci.\ Paris \textbf{205}, 595--598 (1937)

\bibitem{cech} \v{C}ech, E.: Topological Spaces. Wiley, London (1966)

\bibitem{chang} Chang, C.L.: Fuzzy topological spaces. J.\ Math.\ Anal.\ Appl.\ \textbf{24}, 182--190 (1968)

\bibitem{choquet} Choquet, G.: Sur les notions de filtre et de grille. C.R.\ Acad.\ Sci.\ Paris \textbf{224}, 171--173 (1947)

\bibitem{de} De, S.K., Biswas, R., Roy, A.R.: An application of intuitionistic fuzzy sets in medical diagnosis. Fuzzy Sets Syst.\ \textbf{117}, 209--213 (2001)

\bibitem{demir} Demir, \.{I}., Saldam{\i}\c{s}, M., Okurer, M.: BFS-filters and BFS-topological spaces with applications to multi-criteria group decision making. Filomat \textbf{39}(10), 3297--3313 (2025)

\bibitem{dubois} Dubois, D., Prade, H.: Rough fuzzy sets and fuzzy rough sets. Int.\ J.\ Gen.\ Syst.\ \textbf{17}, 191--209 (1990)

\bibitem{hwang} Hwang, C.L., Yoon, K.: Multiple Attribute Decision Making. Springer, Berlin (1981)

\bibitem{hosny} Hosny, M.: Idealization of $j$-approximation spaces. Filomat \textbf{34}(2), 287--301 (2020)

\bibitem{jankovic} Jankovi\'{c}, D.S., Hamlett, T.R.: New topologies from old via ideals. Amer.\ Math.\ Monthly \textbf{97}, 295--310 (1990)

\bibitem{khan} Khan, V.A., Khan, M.D., Arshad, M., Et, M.: Bounded linear operators in neutrosophic normed spaces. Filomat \textbf{39}(2), 515--532 (2025)

\bibitem{kuratowski} Kuratowski, K.: Topologie I. Monografje Matematyczne, Warsaw (1933)

\bibitem{levine} Levine, N.: Semi-open sets and semi-continuity in topological spaces. Amer.\ Math.\ Monthly \textbf{70}, 36--41 (1963)

\bibitem{liu} Liu, F., Zhou, X., Liu, L.: $\mathcal{G}$-quotient mappings and $\mathcal{G}$-continuity in generalized topological spaces. Filomat \textbf{39}(4), 1219--1225 (2025)

\bibitem{lowen} Lowen, R.: Fuzzy topological spaces and fuzzy compactness. J.\ Math.\ Anal.\ Appl.\ \textbf{56}, 621--633 (1976)

\bibitem{mashhour} Mashhour, A.S., Ghanim, M.H., Fath Alla, M.A.: On fuzzy non-continuous mappings. Bull.\ Calcutta Math.\ Soc.\ \textbf{78}, 57--69 (1986)

\bibitem{molodtsov} Molodtsov, D.: Soft set theory---First results. Comput.\ Math.\ Appl.\ \textbf{37}, 19--31 (1999)

\bibitem{njastad} Nj\aa stad, O.: On some classes of nearly open sets. Pacific J.\ Math.\ \textbf{15}, 961--970 (1965)

\bibitem{opricovic} Opricovic, S., Tzeng, G.-H.: Compromise solution by MCDM methods: A comparative analysis of VIKOR and TOPSIS. Eur.\ J.\ Oper.\ Res.\ \textbf{156}, 445--455 (2004)

\bibitem{pawlak} Pawlak, Z.: Rough sets. Int.\ J.\ Comput.\ Inf.\ Sci.\ \textbf{11}, 341--356 (1982)

\bibitem{radzikowska} Radzikowska, A.M., Kerre, E.E.: A comparative study of fuzzy rough sets. Fuzzy Sets Syst.\ \textbf{126}, 137--155 (2002)

\bibitem{roy} Roy, B., Mukherjee, M.N.: On a typical topology induced by a grill. Soochow J.\ Math.\ \textbf{33}(4), 771--786 (2007)

\bibitem{salama_n} Salama, A.A., Alblowi, S.A.: Neutrosophic set and neutrosophic topological spaces. IOSR J.\ Math.\ \textbf{3}(4), 31--35 (2012)

\bibitem{salama_b} Salama, A.S.: Bitopological approximation space with application to data reduction in multi-valued information systems. Filomat \textbf{34}(1), 99--110 (2020)

\bibitem{shabir} Shabir, M., Naz, M.: On soft topological spaces. Comput.\ Math.\ Appl.\ \textbf{61}, 1786--1799 (2011)

\bibitem{smarandache} Smarandache, F.: A Unifying Field in Logics: Neutrosophic Logic. American Research Press (1998)

\bibitem{szmidt} Szmidt, E., Kacprzyk, J.: Intuitionistic fuzzy sets in some medical applications. In: Computational Intelligence.\ Theory and Applications (Fuzzy Days 2001), Lecture Notes in Comput.\ Sci.\ \textbf{2206}, pp.\ 148--151. Springer (2001)

\bibitem{vaidya} Vaidyanathaswamy, R.: The localisation theory in set-topology. Proc.\ Indian Acad.\ Sci.\ \textbf{20}, 51--61 (1945)

\bibitem{yao} Yao, Y.Y.: Constructive and algebraic methods of the theory of rough sets. Inf.\ Sci.\ \textbf{109}, 21--47 (1998)

\bibitem{yigit} Yi\u{g}it, U.: The rough topology for numerical data. Filomat \textbf{39}(17), 6019--6033 (2025)

\bibitem{yuksel} Y\"{u}ksel, \c{S}., Tozlu, N., Dizman, T.H.: An application of multicriteria group decision making by soft covering based rough sets. Filomat \textbf{29}(1), 209--219 (2015)

\bibitem{zadeh} Zadeh, L.A.: Fuzzy sets. Inf.\ Control \textbf{8}, 338--353 (1965)

\end{thebibliography}
\end{document}